%% file: mainAoHiOh.tex
\documentclass[12pt]{article}
\usepackage[dvips]{graphicx}
\usepackage{amsmath,amssymb,amsthm}
\def\Bt{{\bf t}}

\def\By{{\bf y}}

\def\Bb{{\bf b}}

\def\BA{{\bf A}}
\def\BB{{\bf B}}
\def\BC{{\bf C}}
\def\BD{{\bf D}}

\def\BE{{\bf E}}
\def\BF{{\bf F}}
\def\BG{{\bf G}}
\def\BH{{\bf H}}
\def\BI{{\bf I}}
\def\BJ{{\bf J}}
\def\BK{{\bf K}}

\def\BX{{\bf X}}
\def\Bq{{\bf q}}
\def\Bs{{\bf s}}
\def\Bt{{\bf t}}

\def\Bbeta{{\bf \beta}}

\def\Bone{{\bf 1}}

\newtheorem{theorem}{Theorem}[section]
\newtheorem{proposition}[theorem]{Proposition}

\newtheorem{corollary}[theorem]{Corollary}
\newtheorem{remark}[theorem]{Remark}
\newtheorem{question}[theorem]{Question}

\theoremstyle{definition}
\newtheorem{example}[theorem]{Example}

\title{
Markov chain Monte Carlo methods for the 
regular two-level fractional factorial designs and
cut ideals
}
\author{Satoshi Aoki%
\thanks{Graduate School of Science and Engineering (Science Course), Kagoshima University.}%
\ \thanks{JST, CREST.}
, Takayuki Hibi%
\thanks{Department of Pure and Applied Mathematics, Graduate School of Information Science and Technology,
Osaka University.}\ \footnotemark[2]
and Hidefumi Ohsugi%
\thanks{Department of Mathematics, College of Science, Rikkyo University.}\ \footnotemark[2]
}
\date{}

\begin{document}
\maketitle

\begin{abstract}
It is known that
a Markov basis of
the binary graph model of a graph $G$
corresponds to a set of binomial generators of cut ideals $I_{\widehat{G}}$
of the suspension $\widehat{G}$ of $G$.
In this paper, we give another application of cut ideals to statistics.
We show that a set of binomial generators of cut ideals
is a Markov basis of some regular two-level fractional factorial design.
As application, we give a Markov basis of degree 2 for designs
defined by at most two relations. 
\end{abstract}

\input{introduction}

\input{section1}
\input{section2}

\input{section3}

\input{bibliography}

\end{document}

%% file: introduction.tex
\section*{Introduction}
\label{sec:intro}

Application of Gr\"obner bases theory to designed experiments is
one of the main branches 
in a relatively new field in statistics,
called {\it computational algebraic statistics}. 
The first work in this branch is given by Pistone 
and Wynn (\cite{Pistone-Wynn-1996}).
In this paper, they 
presented a method to handle fractional
factorial designs algebraically by defining {\it design ideals}.
As one of the merits to consider design ideals, 
confounding relations 
between the factor effects can be generalized naturally from regular 
to non-regular designs and can be expressed 
concisely by the Gr\"obner bases theory. See 
\cite{Pistone-Wynn-1996} or \cite{Galetto-Pistone-Rogantin-2003} for
details. 
After this work, various algebraic techniques based on the Gr\"obner 
bases theory are applied to the problems of designed experiments
both by algebraists and statisticians. 
For example, an indicator function defined
in \cite{Fontana-Pistone-Rogantin-2000} is a valuable tool to characterize 
non-regular fractional factorial designs. 

On the other hand, there is another main branch in the field of computational
algebraic statistics. In this branch, a key notion is a {\it Markov basis},
which is defined by Diaconis and Sturmfels (\cite{Diaconis-Sturmfels-1998}). 
In this work, they established a procedure for sampling
from discrete conditional distributions by constructing a connected
Markov chain on a given sample space. 
Since this work many papers are given considering Markov bases for various
statistical models, especially for the hierarchical models of multi-dimensional
contingency tables. Intensive results on the structure of Markov bases for
various statistical models are given in 
\cite{Aoki-Hara-Takemura-2012}. 

The arguments in this paper relates to both of the two branches mentioned
above. In fact, the motivation of this paper is 
of interest to investigate statistical problems which are related
to both designed experiments and Markov bases. 
This paper is based on the first works with this motivation, 
\cite{Aoki-Takemura-2010} and \cite{Aoki-Takemura-2009b}. 
In these works, Markov chain Monte Calro methods for testing
factor effects are discussed, when observations are discrete 
and are given in the two-level
or three-level regular fractional factorial designs. As one of the 
contributions of these works, the relation between the 
statistical models for the regular
fractional factorial designs and contingency tables is considered 
through Markov bases. 
As a consequence, to investigate the Markov bases arising in the 
problems of designed experiments, we can refer to the 
known results on the corresponding models for the contingency tables. 
For example, we see that the Markov basis for 
the main effect models of the regular $2^{5-2}_{{\rm III}}$ 
design given by the
defining relation $\BA\BB\BD = \BA\BC\BE = \BI$ is constructed only by the
square-free degree 2 elements (\cite{Aoki-Takemura-2010}). 
This is because the corresponding model 
in the contingency tables is 
the conditional independence model in the $2\times 2\times 2$ table. 
Note that the conditional independence model in the three-way 
contingency table is an example of decomposable models and we know the fact
that a minimal Markov basis for this class of models can be 
constructed only by 
square-free degree 2 elements. See \cite{Dobra-2003} for detail. 

In this paper, following the Markov chain Monte Carlo approach in the
designed experiments by 
\cite{Aoki-Takemura-2010} and \cite{Aoki-Takemura-2009b}, 
we give a new results on the correspondence between the regular
two-level design and the algebraic concept, namely {\it cut ideals}
defined in \cite{Sturmfels-Sullivant-2008}.  
Because the Markov bases are characterized as the generators of 
well-specified toric ideals and 
 are studied not only by statisticians but also by algebraists, 
it is valuable to connect statistical models 
to known class of toric ideals. 
In this paper, we give a fundamental fact that 
the generator of cut ideals can be characterized as 
the Markov bases for the testing problems of log-linear models
for the two-level regular
fractional factorial designs.

The construction of this paper is as follows.
In Section \ref{sec:mcmc} 
we review Markov chain Monte Carlo approach
for testing the fitting of the log-linear models when observations are 
discrete random variables. 
The main results of this paper are given in Section \ref{sec:cut-ideal}.
We show how to relate cut ideals to the fractional factorial designs
and give the main theorem.
In Section \ref{sec:application}, we apply known results on the cut ideals
to the regular fractional factorial designs.

%% file: section1.tex
\section{Markov chain Monte Carlo method for regular two-level
fractional factorial designs}
\label{sec:mcmc}
In this section we introduce Markov chain Monte Carlo methods for 
testing the fitting of the log-linear models for regular two-level
fractional factorial designs with count observations.
Suppose we have nonnegative integer  observations for each run of a
regular fractional design.
For simplicity, we also suppose that the observations are
counts of some events and only one observation is obtained for each
run. 
This is natural for
the settings of Poisson sampling scheme, 
since the set of the totals for each run is the
sufficient statistics for the parameters.
We begin with an example.

\begin{example}[Wave-soldering experiment]
Table \ref{tbl:wave-soldering} 
is a $1/8$ fraction of a full factorial design (i.e., a
$2^{7-3}$ fractional factorial design) defined from the defining relation
\begin{equation}
\BA\BB\BD\BE = \BA\BC\BD\BF = \BB\BC\BD\BG = \BI,
\label{eqn:aliasing-relation-tab1}
\end{equation}
and response data analyzed in \cite{condra-1993} and
reanalyzed in \cite{hamada-nelder-1997jqt}.
\begin{table*}
\label{tbl:wave-soldering}
\begin{center}
\caption{Design and number of defects $y$ for the wave-solder experiment}
\begin{tabular}{ccccccccrrr}\hline
& \multicolumn{7}{c}{Factor} & \multicolumn{3}{c}{$y$}\\
Run & $\BA$ & $\BB$ & $\BC$ & $\BD$ & $\BE$ & $\BF$ & $\BG$ & 
\multicolumn{1}{c}{1} &
 \multicolumn{1}{c}{2} & \multicolumn{1}{c}{3}\\ \hline
 1 & 1 & 1 & 1 & 1 & 1 & 1 & 1 & 13 & 30 &  26\\
 2 & 1 & 1 & 1 & 2 & 2 & 2 & 2 & 4  & 16 &  11\\
 3 & 1 & 1 & 2 & 1 & 1 & 2 & 2 & 20 & 15 &  20\\
 4 & 1 & 1 & 2 & 2 & 2 & 1 & 1 & 42 & 43 &  64\\
 5 & 1 & 2 & 1 & 1 & 2 & 1 & 2 & 14 & 15 &  17\\
 6 & 1 & 2 & 1 & 2 & 1 & 2 & 1 & 10 & 17 &  16\\
 7 & 1 & 2 & 2 & 1 & 2 & 2 & 1 & 36 & 29 &  53\\
 8 & 1 & 2 & 2 & 2 & 1 & 1 & 2 &  5 &  9 &  16\\
 9 & 2 & 1 & 1 & 1 & 2 & 2 & 1 & 29 &  0 &  14\\
10 & 2 & 1 & 1 & 2 & 1 & 1 & 2 & 10 & 26 &   9\\
11 & 2 & 1 & 2 & 1 & 2 & 1 & 2 & 28 &173 &  19\\
12 & 2 & 1 & 2 & 2 & 1 & 2 & 1 &100 &129 & 151\\
13 & 2 & 2 & 1 & 1 & 1 & 2 & 2 & 11 & 15 &  11\\
14 & 2 & 2 & 1 & 2 & 2 & 1 & 1 & 17 &  2 &  17\\
15 & 2 & 2 & 2 & 1 & 1 & 1 & 1 & 53 & 70 &  89\\
16 & 2 & 2 & 2 & 2 & 2 & 2 & 2 & 23 & 22 &   7\\ \hline
\end{tabular}
\end{center}
\end{table*}
In Table \ref{tbl:wave-soldering}, 
the observation $y$ is the number of defects arising in a
wave-soldering process in attaching components to an electronic circuit
card. In Chapter 7 of \cite{condra-1993}, he considered seven factors of a
wave-soldering process: (A) prebake condition, (B) flux density, (C)
conveyer speed, (D) preheat condition, (E) cooling time, (F) ultrasonic
solder agitator and (G) solder temperature, each at two levels with
three boards from each run being assessed for defects.
The aim of this experiment is to decide which levels for each factors
are desirable to reduce solder defects.

Because we only consider designs with a single observation for
each run in this paper, we focus on 
the totals for each run in Table \ref{tbl:wave-soldering}. 
We also ignore the second observation in run 11, which is an obvious outlier as
pointed out in \cite{hamada-nelder-1997jqt}. Therefore the weighted total
of run 11 is $(28 + 19)\times 3/2 = 70.5 \simeq 71$.
By replacing $2$ by $-1$ in Table \ref{tbl:wave-soldering}, we
rewrite $k \times p$ design matrix as $D$, where each element is
$+1$ or $-1$. Consequently, we have 
\[
D = \left(\begin{array}{ccccccc}
+1 & +1 & +1 & +1 & +1 & +1 & +1\\
+1 & +1 & +1 & -1 & -1 & -1 & -1\\
+1 & +1 & -1 & +1 & +1 & -1 & -1\\
+1 & +1 & -1 & -1 & -1 & +1 & +1\\
+1 & -1 & +1 & +1 & -1 & +1 & -1\\
+1 & -1 & +1 & -1 & +1 & -1 & +1\\
+1 & -1 & -1 & +1 & -1 & -1 & +1\\
+1 & -1 & -1 & -1 & +1 & +1 & -1\\
-1 & +1 & +1 & +1 & -1 & -1 & +1\\
-1 & +1 & +1 & -1 & +1 & +1 & -1\\
-1 & +1 & -1 & +1 & -1 & +1 & -1\\
-1 & +1 & -1 & -1 & +1 & -1 & +1\\
-1 & -1 & +1 & +1 & +1 & -1 & -1\\
-1 & -1 & +1 & -1 & -1 & +1 & +1\\
-1 & -1 & -1 & +1 & +1 & +1 & +1\\
-1 & -1 & -1 & -1 & -1 & -1 & -1\\
\end{array}
\right),\ \ 
\By = \left(\begin{array}{c}
69\\ 31\\ 55\\ 149\\ 46\\ 43\\ 118\\ 30\\ 
43\\ 45\\ 71\\ 380\\ 37\\ 36\\ 212\\ 52
\end{array}
\right).
\]
\end{example}

In this paper, we consider designs of $p$ factors with two-level.
We write the observations as $\By = (y_1,\ldots,y_k)'$, 
where $k$ is the run size and $'$ denotes the transpose. 
Write the design matrix $D = (d_{ij})$, where
$d_{ij} \in \{-1,1\}$ is the level of the $j$-th factor in the $i$-th run
for $i=1,\ldots,k, j=1,\ldots,p$.

In this case it is natural to consider the Poisson distribution as the 
sampling model,
in the framework of generalized linear models
(\cite{McCullagh-Nelder-1989}). 
The observations $\By$ are realizations from $k$ Poisson random
variables $Y_1,\ldots,Y_k$, which are mutually independently distributed
with the mean parameter $\mu_i = E(Y_i), i = 1,\ldots,k$. 
We call the log-linear model written by 
\begin{equation}
\log \mu_i = \beta_0 + \beta_i d_{i1} + \cdots + \beta_p d_{ip},\ 
i=1,\ldots,k
\label{eqn:log-linear-null-model}
\end{equation}
as the main effect model in this paper.
The equivalent model in the matrix form is
\[
\left(
\begin{array}{c}
\log\mu_1\\
\vdots\\
\log\mu_k
\end{array}
\right) = M\Bbeta,
\]
where $\Bbeta = (\beta_0,\beta_1,\ldots,\beta_p)'$ and 
\begin{equation}
M = \left(\begin{array}{cccc}
1 &  & & \\
\vdots &  & D & \\
1 &  & & 
\end{array}
\right).
\label{eqn:def-M-by-D}
\end{equation}
We call the $k\times (p+1)$ matrix $M$ 
a {\it model matrix} of the main effect model.
The interpretation of the parameter $\beta_j$ 
in (\ref{eqn:log-linear-null-model}) 
is the parameter contrast for
the main effect of the $j$-th factor.
Following the arguments of \cite{Aoki-Takemura-2010}, we can also
consider the models including various interaction effects. In this paper,
we first describe our methods for the main effect models and will consider
how to treat interaction effects afterward.

To judge the fitting of the main effect 
model (\ref{eqn:log-linear-null-model}), we can perform various 
goodness-of-fit tests. In the goodness-of-fit tests, the main effect model
(\ref{eqn:log-linear-null-model}) is treated as the null model, whereas the 
saturated model is treated as the alternative model. 
Under the null model  (\ref{eqn:log-linear-null-model}), 
$\Bbeta$ is the nuisance parameter and 
the  sufficient statistic for $\Bbeta$
is given by 
$M'\By = (\sum_{i=1}^k y_i, \sum_{i = 1}^k d_{i1}y_i, \ldots, \allowbreak\sum_{i = 1}^k d_{ip}y_i)'$. Then the 
conditional distribution of $\By$ given the sufficient 
statistics is written as
\begin{equation}
f(\By\ |\ M'\By = M'\By^o) = \frac{1}{C(M'\By^o)} \displaystyle\prod_{i = 1}^k\frac{1}{y_i!},
\label{eqn:poisson-conditional-distribution}
\end{equation}
where $\By^o$ is the observation count vector and 
$C(M'\By^o)$ is the normalizing constant determined from
$M'\By^o$ written as
\begin{equation}
 C(M'\By^o) = \displaystyle\sum_{\By \in {\cal F}(M'\By^o)}\left(
\displaystyle\prod_{i = 1}^k\frac{1}{y_i!}
\right),
\label{eqn:poisson-constant}
\end{equation}
and
\begin{equation}
 {\cal F}(M'\By^o) = \{\By\ |\ M'\By = M'\By^o,\ y_i \ \mbox{is a
  nonnegative integer for}\ i = 1,\ldots,k\}.
\label{eqn:poisson-fiber}
\end{equation}
Note that by sufficiency the conditional distribution 
does not depend on the values of the nuisance parameters. 

In this paper we consider various 
goodness-of-fit tests based on the conditional distribution
(\ref{eqn:poisson-conditional-distribution}). There are several ways to
choose the test statistics. For example, the likelihood ratio statistic
\begin{equation}
\label{eq:g2}
T(\By)= G^2(\By) = 2\sum_{i = 1}^{k}y_i\log\frac{y_i}{\hat{\mu_i}}
\end{equation}
is frequently used, 
where $\hat{\mu_i}$ is the maximum likelihood estimate for $\mu_i$
under the null model (i.e., fitted value). 
Note that the traditional asymptotic test evaluates the upper probability
for the observed value $T(\By^o)$ based on the asymptotic distribution
$\chi_{k - p - 1}^2$. However, since the fitting of the asymptotic 
approximation may be sometimes poor, we consider Markov chain Monte Carlo 
methods to evaluate the $p$ values. 
Using the conditional distribution
(\ref{eqn:poisson-conditional-distribution}), the exact $p$ value is
written as
\begin{equation}
 p = \displaystyle\sum_{\By \in {\cal F}(M'\By^o)}f(\By\ |\ M'\By =
 M'\By^o)\Bone(T(\By) \geq T(\By^o)),
\label{eqn:exact-p-value}
\end{equation}
where
\begin{equation}
 \Bone(T(\By) \geq T(\By^o)) 
= \left\{\begin{array}{ll}
1, & \mbox{if}\ T(\By) \geq T(\By^o),\\
0, & \mbox{otherwise}.
\end{array}
\right.
\label{eqn:exact-p-value-1}
\end{equation}
Of course, if we can calculate the exact $p$ value of (\ref{eqn:exact-p-value})
and (\ref{eqn:exact-p-value-1}), it is best. 
Unfortunately, however, an enumeration of all the elements in ${\cal
F}(M'\By^o)$ and hence the calculation of the normalizing constant 
$C(M'\By^o)$ is usually computationally infeasible for large sample
space. Instead, we consider a Markov chain Monte Carlo method. Note that, 
as one of the important advantages of Markov chain Monte Carlo method, 
we need not calculate the normalizing constant 
(\ref{eqn:poisson-constant}) to evaluate $p$ values.

To perform the Markov chain Monte Carlo procedure, 
we have to construct a connected, aperiodic and reversible 
Markov chain over the 
conditional sample space (\ref{eqn:poisson-fiber})
with the stationary distribution (\ref{eqn:poisson-conditional-distribution}).
If such a chain is constructed, we can sample from the chain 
as $\By^{(1)},\ldots,\By^{(T)}$ 
after discarding some initial burn-in steps, and evaluate $p$ values as
\[
\hat{p} = \frac{1}{T}\sum_{t=1}^{T}\Bone(T(\By^{(t)}) \geq T(\By^o)). 
\] 
Such a chain can be constructed easily by {\it Markov basis}. 
Once a Markov basis is
calculated, we can construct a connected, aperiodic and reversible
Markov chain over the space (\ref{eqn:poisson-fiber}), which can be modified
so that the stationary distribution is the conditional distribution
(\ref{eqn:poisson-conditional-distribution})
by the Metropolis-Hastings procedure. See
\cite{Diaconis-Sturmfels-1998} and \cite{Hastings-1970} for details.

Markov basis is characterized algebraically as follows. 
Write indeterminates
$x_1,\ldots,x_k$ and consider polynomial ring $K[x_1,\ldots,x_k]$ for some
field $K$. Consider the integer kernel of the transpose of the model
matrix $M$, ${\rm Ker}_{\mathbb{Z}}M'$. 
For each $\Bb = (b_1,\ldots,b_k)' \in {\rm Ker}_{\mathbb{Z}}M'$, define
binomial in $K[x_1,\ldots,x_k]$ as
\[
f_{\Bb} = \prod_{b_j > 0}x_j^{b_j} - \prod_{b_j < 0}x_j^{-b_j}. 
\]
Then the binomial ideal in $K[x_1,\ldots,x_k]$, 
\[
I(M') = \left<
\{f_{\Bb}\ |\ \Bb \in {\rm Ker}_{\mathbb{Z}}M'\}
\right>,
\]
is called a toric ideal with the configuration $M'$. 
Let $\{f_{\Bb^{(1)}},\ldots,f_{\Bb^{(s)}}\}$ be any generating set of $I(M')$.
Then the set of integer vectors $\{\Bb^{(1)},\ldots,\Bb^{(s)}\}$ constitutes
a Markov basis. See \cite{Diaconis-Sturmfels-1998} for detail. 
To compute a Markov basis for given configuration $M'$, we can rely on
various algebraic softwares such as 4ti2 (\cite{4ti2}). 
See the following example. 

\begin{example}[Wave-soldering experiment, continued]
We analyze the data in Table \ref{tbl:wave-soldering}.
The fitted value under the main effect model is calculated as
\[\begin{array}{l}
\hat{\mu} = (68.87,\ 19.70,\ 78.85,\ 147.59,\ 12.14,\ 54.77,\ 104.53,\ 
54.54,\\ 
\hspace*{10mm}75.31,\ 39.29,\ 75.00,\ 338.37,\ 27.83,\ 52.09,\ 208.47,\ 
59.64)'.
\end{array}
\] 
Then the likelihood ratio for the observed data 
is calculated as $T(\By^o) = G^2(\By^o) = 117.81$ and the
corresponding asymptotic $p$ value is less than $0.0001$ from the asymptotic 
distribution $\chi_8^2$. This result tells us that the null hypothesis
is highly significant and is rejected, i.e., the existence of some interaction
effects is suggested. To evaluate the $p$ value by Markov chain Monte Carlo
method, we have to calculate a Markov basis first. If we use 4ti2, we prepare
the data file (configuration $M'$) as
\begin{verbatim}
8 16 
 1  1  1  1  1  1  1  1  1  1  1  1  1  1  1  1
 1  1  1  1  1  1  1  1 -1 -1 -1 -1 -1 -1 -1 -1
 1  1  1  1 -1 -1 -1 -1  1  1  1  1 -1 -1 -1 -1
 1  1 -1 -1  1  1 -1 -1  1  1 -1 -1  1  1 -1 -1
 1 -1  1 -1  1 -1  1 -1  1 -1  1 -1  1 -1  1 -1
 1 -1  1 -1 -1  1 -1  1 -1  1 -1  1  1 -1  1 -1
 1 -1 -1  1  1 -1 -1  1 -1  1  1 -1 -1  1  1 -1
 1 -1 -1  1 -1  1  1 -1  1 -1 -1  1 -1  1  1 -1
\end{verbatim}
and run the command {\tt markov}. Then we have a minimal Markov basis 
with $77$ elements as follows.
\begin{verbatim}
77 16
 0  0  0  0  0  0  0  0  1  1 -1 -1 -1 -1  1  1 
 0  0  0  0  0  1 -1  0  1  0  0 -1 -1 -1  1  1 
 0  0  0  0  0  1  0 -1  0  1  0 -1 -1 -1  1  1 
 0  0  0  0  1  0 -1  0  1  0 -1  0 -1 -1  1  1 
 0  0  0  0  1  0  0 -1  0  1 -1  0 -1 -1  1  1 
 0  0  0  0  1  1 -1 -1  0  0  0  0 -1 -1  1  1 
 0  0  0  1  0  0 -1  0  1  0 -1 -1  0 -1  1  1 
.....
\end{verbatim}
Using this Markov basis, we can evaluate $p$ value by Markov chain Monte Carlo
method. After $50,000$ burn-in-steps from $\By^o$ itself as the initial state, 
we sample $100,000$ Monte Carlo sample by Metropolis-Hasting algorithm,
which yields $\hat{p} = 0.0000$ again. 
Figure \ref{fig:chi2-sample-example} is a histogram of the Monte Carlo sampling
of the likelihood ratio statistic under the main effect model, along with the
corresponding asymptotic distribution $\chi^2_8$.
\begin{figure*}[htbp]
\begin{center}
\includegraphics[height=8cm,width=10cm]{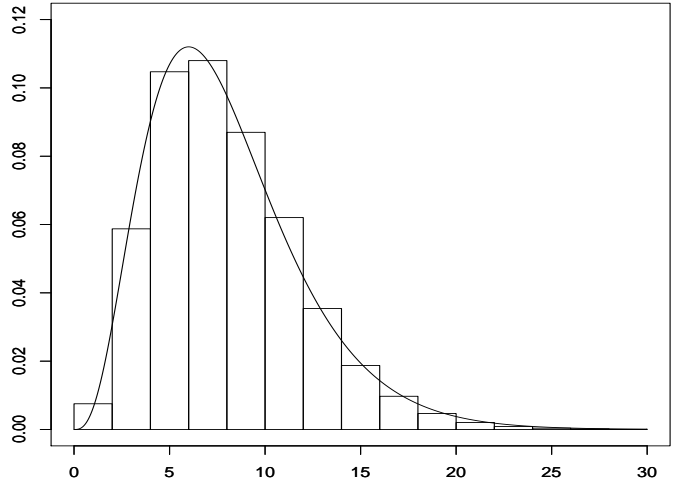}
\caption{Asymptotic and Monte Carlo estimated distribution of likelihood ratio}
\label{fig:chi2-sample-example}
\end{center}
\end{figure*}
\end{example}

If the fitting of the main effect model is poor, or we have some prior
knowledge on the existence of interaction effects, we consider the models
including interaction effects. The models including the interaction effects
are also described by the model matrix. 
The modeling method presented in 
\cite{Aoki-Takemura-2010} is as follows. If we want to consider the models
including interaction effects, add columns to the model matrix of the main 
effect model (\ref{eqn:def-M-by-D}) so that the corresponding parameter
can be interpreted as the parameter contrast for the additional interaction
effect. We describe this point by the previous example. See
\cite{Aoki-Takemura-2010} for details.

\begin{example}[Wave-soldering experiment, continued]
\label{example:wave-soldering-cont-cont}
As is pointed out in \cite{hamada-nelder-1997jqt}, the existence of 
some interaction effects is suggested for the data in 
Table \ref{tbl:wave-soldering}. 
In \cite{hamada-nelder-1997jqt}, the model including $2$ two-factor 
interaction effects, $\BA\times \BC$ and $\BB\times \BD$ is considered. 
Here we call this model as $M_2$ model.
The model matrix of the $M_2$ model is constructed by adding two columns,
\[
\left(
\begin{array}{rrrrrrrrrrrrrrrr}
1 & 1 & -1 & -1 & 1 & 1 & -1 & -1 & -1 & -1 & 1 & 1 & -1 & -1 & 1 & 1\\
1 & -1 & 1 & -1 & -1 & 1 & -1 & 1 & 1 & -1 & 1 & -1 & -1 & 1 & -1 & 1
\end{array}
\right)'
\]
to the model matrix of the main effect model.
Note that the above two columns are the element-wise product of 
the (first, third)
columns and the (second, fourth) columns of $D$, respectively.
Using this model matrix (or configuration matrix as its transpose), we can
investigate the fitting of the $M_2$ model in the similar way. 
The fitted value under the $M_2$ model is
\[\begin{array}{l}
\hat{\mu} = (64.53,\ 47.25,\ 53.15,\ 151.08,\ 30.43,\ 46.79,\ 115.24,\ 
32.53,\\ 
\hspace*{10mm}49.42,\ 46.13,\ 70.90,\ 360.54,\ 35.19,\ 30.26,\ 232.14,\ 
51.42)'.
\end{array}
\] 
Then the likelihood ratio for the observed data 
is calculated as $T(\By^o) = G^2(\By^o) = 19.0927$ and the
corresponding asymptotic $p$ value is $0.00401$ from the asymptotic 
distribution $\chi_6^2$. 
Though this result still suggests the significantly poor fitting of the
$M_2$ model, much larger $p$ value and much smaller likelihood
ratio than those of the main effect model tell us that $M_2$ model is
much better than the main effect model. 
We have the Markov chain Monte Carlo estimate
of $p$ value as $\hat{p} = 0.0031$ from $100,000$ Monte Calro sample
after $50,000$ burn-in-steps.
\end{example}

An important point here is that the model matrix for the models
with interaction effects constructed in this way
is equivalent to the model matrix for the main effect model for some
regular fractional factorial design of resolution III or more. 
For example, the model matrix for the $M_2$ model in Example 
\ref{example:wave-soldering-cont-cont} is equivalent to 
the main effect model
for the $2^{9-5}$ fractional factorial design defined from
\[
\BA\BB\BD\BE = \BA\BC\BD\BF = \BB\BC\BD\BG = \BA\BC\BH = \BB\BD\BJ = \BI.
\]
This relation is given by adding $\BH=\BA\BC$ and $\BJ=\BB\BD$ to 
(\ref{eqn:aliasing-relation-tab1}), as if there exists two additional
factors $\BH$ and $\BJ$ in Table \ref{tbl:wave-soldering}. 
In this paper, we consider the relation between the model matrix 
and the cut ideals. From the above considerations, 
we can restrict our attentions to the main effect models to
consider the relations to the cut ideals.
We give the relation in Section \ref{subsec:design-cut-ideal}.
We will consider the models including interaction effects again in 
Section \ref{subsec:classification} 
to consider the relation from the
practical viewpoint.

%% file: section2.tex
\section{Two-level regular fractional factorial designs and cut ideals}
\label{sec:cut-ideal}

In this section, we show that 
a cut ideal for a finite connected graph can be 
characterized as 
the toric ideal $I(M')$ for a model matrix of 
the main effect model for some regular two-level fractional factorial designs.

\subsection{Cut ideals}
We start with the definition of the cut ideal. Consider a connected finite
graph $G = (V, E)$. We also consider unordered partitions $A|B$ of the
vertex set $V$. Let ${\cal P}(V)$ be the set of the unordered partitions
of $V$, i.e., 
\[
{\cal P}(V) = \{A|B \ |\ A\cup B = V,\ A\cap B = \emptyset\}.
\]
We introduce the sets of indeterminates
$\{s_{ij} \ | \ \{i,j\} \in E\}$, 
$\{t_{ij} \ | \{i,j\}\in E\}$ and 
$\{q_{A|B} \ | \ A|B \in {\cal P}(V) \}$.
Let
\[
K[\Bq] = K[q_{A|B}\ |\ A|B \in {\cal P}(V)],
\]
\[
K[\Bs, \Bt] = K[s_{ij},t_{ij}\ |\ \{i,j\} \in E]
\]
be polynomial rings over a field $K$.
For each partition $A|B \in {\cal P}(V)$, we define a subset ${\rm Cut}(A|B)$
of the edge set $E$ as
\[
{\rm Cut}(A|B) = \{\{i,j\} \in E \ |\ i \in A, j \in B\ {\rm or}\ i \in B, j \in A\}.
\]
Define homomorphism of polynomial rings as
\begin{equation}
\phi_G:\ K[\Bq] \rightarrow K[\Bs, \Bt],\ \ 
q_{A|B} \mapsto \prod_{\{i,j\}\in {\rm Cut}(A|B)}s_{ij}
\ \cdot \prod_{\{i,j\}\in E\setminus{\rm Cut}(A|B)}t_{ij}.
\label{eqn:def-of-cut-ideal-map}
\end{equation}
We may think of $\Bs$ and $\Bt$ as abbreviations for 
``separated'' and ``together'', respectively. Then the cut ideal of the graph
 $G$ is defined as $I_G = {\rm Ker} (\phi_G)$. 
We also use 
the following two examples given in \cite{Sturmfels-Sullivant-2008}.
\begin{example}[Complete graph on four vertices]
\label{example:K4}
Let $G = K_4$ be the complete graph on four vertices $V=\{1,2,3,4\}$. Then
the edge set is $E = \{12,13,14,23,24,34\}$. 
The map $\phi_{K_4}$ is specified by
\[
\begin{array}{rcl}
q_{\emptyset | 1234} & \mapsto & t_{12}t_{13}t_{14}t_{23}t_{24}t_{34}\\
q_{1 | 234} & \mapsto & s_{12}s_{13}s_{14}t_{23}t_{24}t_{34}\\
q_{2 | 134} & \mapsto & s_{12}t_{13}t_{14}s_{23}s_{24}t_{34}\\
q_{3 | 124} & \mapsto & t_{12}s_{13}t_{14}s_{23}t_{24}s_{34}\\
q_{4 | 123} & \mapsto & t_{12}t_{13}s_{14}t_{23}s_{24}s_{34}\\
q_{12 | 34} & \mapsto & t_{12}s_{13}s_{14}s_{23}s_{24}t_{34}\\
q_{13 | 24} & \mapsto & s_{12}t_{13}s_{14}s_{23}t_{24}s_{34}\\
q_{14 | 23} & \mapsto & s_{12}s_{13}t_{14}t_{23}s_{24}s_{34}.
\end{array}
\]
In this case, the cut ideal is a principal ideal
given by
\[
I_{K_4} = \left<
q_{\emptyset | 1234}q_{12|34}q_{13|24}q_{14|23} - 
q_{1|234}q_{2|134}q_{3|124}q_{4|123}
\right>.
\]
\end{example}

\begin{example}[$4$-cycle]
\label{example:4-cycle}
Let $G = C_4$ be the $4$-cycle with
\[
V = \{1,2,3,4\},\ E=\{12,23,34,14\}.
\]
The map $\phi_{C_4}$ is derived from $\phi_{K_4}$ in Example \ref{example:K4}
by setting
\[
s_{13} = t_{13} = s_{24} = t_{24} = 1
\]
as
\[
\begin{array}{rcl}
q_{\emptyset | 1234} & \mapsto & t_{12}t_{14}t_{23}t_{34}\\
q_{1 | 234} & \mapsto & s_{12}s_{14}t_{23}t_{34}\\
q_{2 | 134} & \mapsto & s_{12}t_{14}s_{23}t_{34}\\
q_{3 | 124} & \mapsto & t_{12}t_{14}s_{23}s_{34}\\
q_{4 | 123} & \mapsto & t_{12}s_{14}t_{23}s_{34}\\
q_{12 | 34} & \mapsto & t_{12}s_{14}s_{23}t_{34}\\
q_{13 | 24} & \mapsto & s_{12}s_{14}s_{23}s_{34}\\
q_{14 | 23} & \mapsto & s_{12}t_{14}t_{23}s_{34}.
\end{array}
\]
In this case, the cut ideal is given by
\[
 I_{C_4} = \left<
q_{\emptyset | 1234}q_{13|24} - q_{1|234}q_{3|124},\ 
q_{\emptyset | 1234}q_{13|24} - q_{2|134}q_{4|123},\ 
q_{\emptyset | 1234}q_{13|24} - q_{12|34}q_{14|23}
\right>.
\]
\end{example}

Now we relates the cut ideals to the regular two-level fractional factorial
designs. 
We express the map $\phi_G$ by $2^{|V|-1} \times 2|E|$ 
matrix $H = \{h_{A|B,e}\}$ 
where
each row of $H$ represents  $A|B \in {\cal P}(V)$ 
and each two columns of $H$ represents $E$ as
\[
h_{A|B,e} = \left\{\begin{array}{ll}
(1,0) & \mbox{if}\ e \in E\setminus {\rm Cut}(A|B)\\ 
(0,1) & \mbox{if}\ e \in {\rm Cut}(A|B).
\end{array}
\right.
\]
Note that there are $|{\cal P}(V)| = 2^{|V|-1}$ unordered partitions of $V$.
We also see that each two columns of $H$ correspond to $\Bt$ and $\Bs$. 
Then the cut ideal, the kernel of $\phi_G$ of 
(\ref{eqn:def-of-cut-ideal-map}), is written as the toric ideal of
the configuration matrix $H'$.

\begin{example}[$4$-cycle, continued]
For the case of $G = C_4$ of Example \ref{example:4-cycle}, 
the matrix $H$ can be written as follows.
\begin{equation}
\begin{array}{c|cccccccc|}
\multicolumn{1}{c}{} & t_{12} & s_{12} & t_{14} & s_{14} & 
t_{23} & s_{23} & t_{34} & \multicolumn{1}{c}{s_{34}}\\ \cline{2-9}
q_{\emptyset | 1234} & 1 & 0 & 1 & 0 & 1 & 0 & 1 & 0\\
q_{3|124} & 1 & 0 & 1 & 0 & 0 & 1 & 0 & 1\\
q_{4|123} & 1 & 0 & 0 & 1 & 1 & 0 & 0 & 1\\
q_{12|34} & 1 & 0 & 0 & 1 & 0 & 1 & 1 & 0\\
q_{14|23} & 0 & 1 & 1 & 0 & 1 & 0 & 0 & 1\\
q_{2|134} & 0 & 1 & 1 & 0 & 0 & 1 & 1 & 0\\
q_{1|234} & 0 & 1 & 0 & 1 & 1 & 0 & 1 & 0\\
q_{13|24} & 0 & 1 & 0 & 1 & 0 & 1 & 0 & 1\\ \cline{2-9}
\end{array}
\label{eqn:table-q-st}
\end{equation}
\end{example}
The kernel of $H'$ coincides to the kernel of $M'$ of (\ref{eqn:def-M-by-D})
for the two-level design $D$ of $|E|$ factors with $2^{|V|-1}$ runs, where
the level of the factor $X_e$ for the run $A|B \in {\cal P}(V)$ is given
by the following map:
\begin{equation}
\begin{array}{cccc}
 X_e : & {\cal P}(V)& \rightarrow & \{+1,-1\}\\
  & \rotatebox{90}{$\in$} & &\rotatebox{90}{$\in$} \\
  & A|B &\mapsto & \left\{\begin{array}{ll}
+1 & \mbox{if}\ e \in E \setminus {\rm Cut}(A|B)\\
-1 & \mbox{if}\ e \in {\rm Cut}(A|B)
\end{array}
\right.
\end{array}
\label{eqn:def-X_e}
\end{equation}

\begin{example}[$4$-cycle, continued]
\label{example:4-cycle-relation}
For the case of $G = C_4$, the map $X_e$ of (\ref{eqn:def-X_e}) gives
the design matrix $D$ as follows.
\[
\begin{array}{c|rrrr|}
\multicolumn{1}{c}{} & X_{12} & X_{14} & X_{23} & \multicolumn{1}{c}{X_{34}}\\ \cline{2-5}
q_{\emptyset | 1234} &  1 & 1 & 1 & 1\\
q_{3|124}            &  1 & 1 &-1 &-1\\
q_{4|123}            &  1 &-1 & 1 &-1\\
q_{12|34}            &  1 &-1 &-1 & 1\\
q_{14|23}            & -1 & 1 & 1 &-1\\
q_{2|134}            & -1 & 1 &-1 & 1\\
q_{1|234}            & -1 &-1 & 1 & 1\\
q_{13|24}            & -1 &-1 &-1 &-1\\ \cline{2-5}
\end{array}
\]
For this $D$, it is easily seen that ${\rm Ker}(M')$ coincides to 
${\rm Ker}(H')$ if $H$ is given by (\ref{eqn:table-q-st}).
\end{example}

\subsection{Regular designs and cut ideals}
\label{subsec:design-cut-ideal}
In Example \ref{example:4-cycle-relation}, we obtain
the toric ideal for the main effect model of the regular two-level
fractional factorial designs 
defined by $X_{12}X_{14}X_{23}X_{34}=1$ from the 
the cut ideal of $G = C_4$. 
In fact, there is a clear relation between 
finite connected graphs $G$ and regular two-level designs $D$. 
As we have seen in Example \ref{example:4-cycle-relation}, the cut ideal 
for $G$ can be related to the design of $p = |E|$ factors with 
$k = 2^{|V|-1}$ runs. 
Since each factor of this design corresponds to the edge $E$ of $G$, we
write each factor $X_{ij}$ for $\{i,j\} \in E$.
Since there are $2^p$ runs in the full factorial
design of $p$ factors, the design obtained from $G$ by the relation 
(\ref{eqn:def-X_e}) is a $2^{|V|-1-p}$ fraction of the full factorial
design of $p$ factors. We show this fraction is specified as the regular
fractional factorial designs.

Let $G=(V,E)$ be a finite connected graph with the edge set $E = \{e_1, \ldots, e_p\}$.
Then, the {\em cycle space} $\mathcal{C} (G)$ of $G$ is a subspace of 
 $\mathbb{F}_2^{|E|}$ spanned by 
 $$
\left\{ 
\left.
 {\bf e}_{i_1} + \cdots +{\bf e}_{i_r} \in \mathbb{F}_2^{|E|} \ \right| \ 
(e_{i_1}, \ldots,e_{i_r}) \mbox{ is a cycle of } G \right\},
 $$
where ${\bf e}_j$ is the $j$th coordinate vector of $\mathbb{F}_2^{|E|}$.
On the other hand, the {\em cut space} $\mathcal{C}^*(G)$ of $G$ is a subspace of 
 $\mathbb{F}_2^{|E|}$ defined by 
 $$
\mathcal{C}^*(G)
=
\left\{
\left.  \sum_{e_j \in   {\rm Cut}(A |B)  } {\bf e}_j \in \mathbb{F}_2^{|E|} \ \right| \ 
A|B \in \mathcal{P} (V) 
\right\}.
$$
Fix a spanning tree $T$ of $G$.
For each $e \in E \setminus T$, the set $T \cup \{e\}$ has exactly one cycle $C_e$ of $G$.
Such a cycle $C_e$ is called a {\em fundamental cycle} of $G$.
Since $T$ has $|V|-1$ edges, there are $ |E| - |V| +1$ edges in $E \setminus T$.
It then follows that there exists $ |E| - |V| +1$ fundamental cycles in $G$.
The following proposition is known in graph theory \cite{Diestel}:

\begin{proposition}
\label{graphtheory}
Let $G=(V,E)$ be a finite connected graph.
Then, we have the following:
\begin{itemize}
\item[{\rm (i)}]
$\mathcal{C}^* (G) = \mathcal{C} (G)^\perp 
\left(=
\left\{
\left.
 {\bf v} \in \mathbb{F}_2^{|E|}
\ \right| \ 
{\bf v} \cdot {\bf u} = 0 \mbox{ for all } {\bf u} \in \mathcal{C} (G)
\right\} \right);
$
\item[{\rm (ii)}]
Given spanning tree of $G$,
the cycle space $\mathcal{C} (G)$ is spanned by 
 $$
\left\{
\left.
 {\bf e}_{i_1} + \cdots +{\bf e}_{i_r} \in \mathbb{F}_2^{|E|} \ \right| \ 
(e_{i_1}, \ldots,e_{i_r}) \mbox{ is a fundamental cycle of } G \right\};
 $$
\item[{\rm (iii)}]
$\dim \mathcal{C} (G) = |E| - |V| +1$ and 
$\dim \mathcal{C}^* (G) = |V| -1$.
\end{itemize}
\end{proposition}

By Proposition \ref{graphtheory}, we have the following:
 
\begin{theorem}
Let $G = (V, E)$ be a finite connected graph and 
let $D$ be the design matrix
of $|E|$ factors with $2^{|V|-1}$ runs defined by 
(\ref{eqn:def-X_e}).
Then $D$ is a regular fractional factorial design with all relations
\begin{equation}
X_{e_{i_1}}(A|B)X_{e_{i_2}}(A|B)\cdots X_{e_{i_m}}(A|B) = 1,
\label{eqn:defining-relation-cycle}
\end{equation}
where 
$(e_{i_1},\ldots,e_{i_m})$ is a fundamental cycle of $G$. 
\label{thm:main-theorem}
\end{theorem}

\begin{proof}
The equation (\ref{eqn:defining-relation-cycle}) is equivalent to 
the equation
$$
({\bf e}_{i_1} + \cdots +{\bf e}_{i_m} ) \cdot 
\left( \sum_{e_j \in   {\rm Cut}(A |B)  } {\bf e}_j \right) = 0
$$
in $ \mathbb{F}_2^{|E|}$.
Thus, the assertion follows from  (i) and (ii) of Proposition \ref{graphtheory}.
\end{proof}

It may be a helpful to see a typical example.

\begin{example}
Consider $G = (V, E)$ given by 
\[
V=\{1,2,3,4,5,6\},\ E=\{12,13,23,34,45,56,26\}.
\]
\begin{center}
\begin{picture}(70,70)(0,0)
\put(0,30){\circle*{5}}
\put(15,0){\circle*{5}}
\put(30,30){\circle*{5}}
\put(30,60){\circle*{5}}
\put(45,0){\circle*{5}}
\put(60,30){\circle*{5}}
\put(15,0){\line(1,0){30}}
\put(45,0){\line(1,2){15}}
\put(0,30){\line(1,-2){15}}
\put(0,30){\line(1,1){30}}
\put(30,60){\line(1,-1){30}}
\put(30,30){\line(0,1){30}}
\put(30,30){\line(1,-2){15}}
\put(30,30){\line(-1,-2){15}}
\put(35,34){$1$}
\put(13,-10){$2$}
\put(43,-10){$3$}
\put(65,30){$4$}
\put(35,60){$5$}
\put(-10,30){$6$}
\end{picture}
\end{center}
From (\ref{eqn:defining-relation-cycle}), we see that 
$G$ corresponds to $2^{7-3}$ design with
\[
X_{12}X_{13}X_{23} = X_{13}X_{34}X_{45}X_{15} = X_{12}X_{26}X_{56}X_{15} = 1.
\]
Note that the relations corresponding 
the dependent cycles such as 
$\{23,34,45,56,26\}$ can be  
derived as
\[
X_{23}X_{34}X_{45}X_{56}X_{26} = 
(X_{12}X_{13}X_{23})(X_{13}X_{34}X_{45}X_{15})(X_{12}X_{26}X_{56}X_{15}) = 1.
\]
\end{example}

Theorem \ref{thm:main-theorem} shows the relation of the cut ideals
and regular two-level fractional factorial designs. 
For a given connected finite graph, 
we can consider corresponding 
regular two-level fractional factorial designs
from Theorem \ref{thm:main-theorem}. 
Unfortunately, however, 
the converse does not always hold. For given 
regular two-level fractional factorial designs (strictly, we should say
that ``for given designs and {\it models}'', which we consider 
in Section \ref{subsec:classification}), it does not always exist
corresponding connected finite graphs. 

\begin{proposition}
If a $2^{p-q}$ design
corresponds to a finite graph by the relation (\ref{eqn:def-X_e}),
then we have
$p \leq {p-q+1 \choose{2}}$.
\end{proposition}

\begin{proof}
Corresponding connected graphs $G=(V, E)$ does not exist because $|E| = p$ and
$|V| = p-q+1$ must be satisfied if it exists. (There are 
${p-q+1 \choose{2}}$ edges in $K_{p-q+1}$.) 
\end{proof}

Thus, obvious counterexamples for the converse 
are given
since
some regular $2^{p-q}$ designs satisfy ${p-q+1 \choose{2}} < p$
(for example, $(p, q)=(5,3), (5,4),(6,4),(6,5)$ and so on). 
On the other hand, a necessary condition related with the resolution is as follows.

\begin{proposition}
\label{trifreegraph}
If a $2^{p-q}$ design of resolution ${\rm IV}$ or more
corresponds to a finite graph by the relation (\ref{eqn:def-X_e}),
then we have
$p \leq \lfloor (p-q+1)^2/4 \rfloor$.
\end{proposition}

\begin{proof}
Mantel's theorem in graph theory says that
the number of edges in triangle-free graph with $n$ vertices is
at most $\lfloor n^2/4 \rfloor$.
\end{proof}

If the resolution of a design is ${\rm V}$ or more,
then similar results are obtained by the results in \cite{GKL}.
From these considerations, 
an important question arises.

\begin{question}
\label{question}
Characterize regular two-level fractional factorial designs that can
correspond to a finite graph by the relation (\ref{eqn:def-X_e}).
\end{question}

A complete answer to this question is not yet obtained at present.
We give results for $8$ runs and $16$ runs designs 
in Section \ref{subsec:classification}. 
We present several fundamental characterizations in the rest of this section.
Note that the above 
correspondence is not one-to-one even if it exists. 
In fact, for any finite connected graph 
$G$, we can specify a design $D$ uniquely by 
(\ref{eqn:defining-relation-cycle}). However, for a given design $G$, 
we can consider several graphs satisfying the relation 
(\ref{eqn:defining-relation-cycle}) if it exists. 

\begin{example}[$2^{5-1}$ design with $X_{12}X_{13}X_{23} = 1$ of $5$ factors]
Consider $2^{5-1}$ fractional factorial design $X_{12}X_{13}X_{23} = 1$ of
$5$ factors,
or, $\BA\BB\BC = \BI$ in the convention of designed experiment literature.
There are several corresponding graphs that give this design such as follows.
\begin{center}
\begin{picture}(180,70)(0,0)
\put(0,25){\line(2,1){50}}
\put(0,25){\line(2,-1){50}}
\put(50,0){\line(0,1){50}}
\put(50,50){\line(1,0){50}}
\put(100,50){\line(1,0){50}}
\put(0,25){\circle*{5}}
\put(50,0){\circle*{5}}
\put(50,50){\circle*{5}}
\put(100,50){\circle*{5}}
\put(150,50){\circle*{5}}
\put(-3,35){$1$}
\put(47,55){$2$}
\put(55,0){$3$}
\put(97,55){$4$}
\put(147,55){$5$}
\end{picture}
\begin{picture}(180,70)(0,0)
\put(0,25){\line(2,1){50}}
\put(0,25){\line(2,-1){50}}
\put(50,0){\line(0,1){50}}
\put(50,50){\line(1,0){50}}
\put(50,0){\line(1,0){50}}
\put(0,25){\circle*{5}}
\put(50,0){\circle*{5}}
\put(50,50){\circle*{5}}
\put(100,50){\circle*{5}}
\put(100,0){\circle*{5}}
\put(-3,35){$1$}
\put(47,55){$2$}
\put(57,5){$3$}
\put(97,55){$4$}
\put(97,5){$5$}
\end{picture}
\end{center}
Later, we will be able to understand this by Proposition \ref{toricfiber}.
(Both of two graphs are 0-sum of the same pair of graphs.)
\end{example}

Now we show two important special cases, 
designs corresponding to complete graphs
and trees.

\begin{example}[Complete graph on four vertices (continued)]
\label{example:K4-cont}
Let $G = K_4$ be the complete graph on four vertices $V=\{1,2,3,4\}$ 
in Example \ref{example:K4}.
\begin{center}
\begin{picture}(55,55)(0,0)
\put(0,0){\line(1,0){40}}
\put(0,0){\line(0,1){40}}
\put(0,0){\line(1,1){40}}
\put(40,0){\line(0,1){40}}
\put(0,40){\line(1,0){40}}
\put(0,40){\line(1,-1){40}}
\put(0,00){\circle*{5}}
\put(0,40){\circle*{5}}
\put(40,0){\circle*{5}}
\put(40,40){\circle*{5}}
\put(-10,40){$1$}
\put(43,40){$2$}
\put(-10,0){$4$}
\put(43,0){$3$}
\end{picture}
\end{center}

From (\ref{eqn:defining-relation-cycle}), we have the following 
regular $2^{6-3}$ fractional factorial design.
\[
\begin{array}{c|rrrrrr|}
\multicolumn{1}{c}{} & X_{12} & X_{13} & X_{14} & X_{23} & X_{24} & 
\multicolumn{1}{c}{X_{34}}\\ \cline{2-7}
q_{\emptyset | 1234} &  1 & 1 & 1 & 1 & 1 & 1\\
q_{3|124}            &  1 &-1 & 1 &-1 & 1 &-1\\
q_{4|123}            &  1 & 1 &-1 & 1 &-1 &-1\\
q_{12|34}            &  1 &-1 &-1 &-1 &-1 & 1\\
q_{14|23}            & -1 &-1 & 1 & 1 &-1 &-1\\
q_{2|134}            & -1 & 1 & 1 &-1 &-1 & 1\\
q_{1|234}            & -1 &-1 &-1 & 1 & 1 & 1\\
q_{13|24}            & -1 & 1 &-1 &-1 & 1 &-1\\ \cline{2-7}
\end{array}
\]
The defining relation of this design is 
\[
X_{12}X_{13}X_{23} = X_{12}X_{14}X_{24} = X_{13}X_{14}X_{34} = 1,
\]
where the three terms 
$X_{12}X_{13}X_{23}, X_{12}X_{14}X_{24}, X_{13}X_{14}X_{34}$ correspond
 to the independent cycle of $K_4$. 
\end{example}
Generalizing Example \ref{example:K4-cont}, we summarize the following
important cases.
\begin{corollary}
Let $G = K_n$ be the complete graph on $|V| = n$ vertices.
Then, $G$ is specified as the
regular 
$2^{c_1 - c_2}$ fractional factorial design of $c_1$ two-level 
factors by (\ref{eqn:defining-relation-cycle}), where
\[
c_1 = {n \choose{2}},\ \ c_2 = {n-1 \choose{2}}.
\]
The defining relation of this design is written 
as $X_{1 i}X_{1 j}X_{i j} = 1$ for any pair 
$(i , j)$ with $2 \leq i < j \leq n$.
\end{corollary}

Another important case is as follows.
\begin{corollary}
Any spanning tree $G = (V, E)$ is specified as the full factorial design of
$|V|-1$ two-level factors by (\ref{eqn:defining-relation-cycle}). 
\end{corollary}

\subsection{Models for the designs with $8$ runs and $16$ runs}
\label{subsec:classification} 

It seems very difficult to answer Question \ref{question} in general.
As special cases, 
we show that 
$2^{p-1}$ and $2^{p-2}$ fractional factorial designs can relate to
graphs from algebraic theories in 
Section \ref{sec:application}.
As another approach, we investigate all the practical models arising 
in the two-level fractional
factorial designs with $8$ runs and $16$ runs.
In Section \ref{subsec:classification}, we consider models with 
interaction effects to classify the cases arising in applications.

First we consider models for the designs with $8$ runs.
Among the regular designs with $8$ runs, 
the most frequently used designs are listed in
Table \ref{tbl:list-8-runs-designs}.
\begin{table*}
\begin{center}
\caption{$2^{p-q}$ fractional factorial designs with $8$ runs ($p - q = 3$)}
\label{tbl:list-8-runs-designs}
\begin{tabular}{ccl}\hline
Number of factors $p$ & Resolution & Design Generators\\ \hline
4 & IV & $\BD = \BA\BB\BC$ \\ 
5 & III & $\BD = \BA\BB,\ \BE = \BA\BC$ \\
6 & III & $\BD = \BA\BB,\ \BE = \BA\BC,\ \BF = \BB\BC$ \\
7 & III & $\BD = \BA\BB,\ \BE = \BA\BC,\ \BF = \BB\BC, \BG =
 \BA\BB\BC$\\ \hline
\end{tabular}
\end{center}
\end{table*}
We ignore $2^{7-4}_{{\rm III}}$ design in Table \ref{tbl:list-8-runs-designs}
since the main effect model is saturated and cannot be tested in our method.
For the other $3$ designs, $2^{4-1}_{\rm IV}, 2^{5-2}_{\rm III}$ and 
$2^{6-3}_{\rm III}$ designs, we consider models to be tested. 

For $2^{4-1}_{\rm IV}$ design of $\BA\BB\BC\BD=\BI$, 
we can consider $3$ models as follows.
\begin{itemize}
\item Main effect model:\\
We write it as $\BA/\BB/\BC/\BD$.
\item Models with all the main effects and $1$ two-factor interaction effect:\\
Without loss of generality, we consider $\BA\times \BB$ as the interaction 
effect. We write it as $\BA\BB/\BC/\BD$.
\item Models with all the main effects and $2$ two-factor interaction 
effects:\\
Note that we cannot consider models including, say, $\BA\times \BB$ and 
$\BC\times \BD$, because they are confounded.
Without loss of generality, we consider $\BA\times \BB$ and $\BA\times \BC$ 
as the two interaction effects. 
We write it as $\BA\BB/\BA\BC/\BD$.
\end{itemize}
Note that we cannot consider models with more than $2$ two-factor interactions
because there are $8$ parameters in the saturated models for the design 
with $8$ runs.
For these models, we can construct model matrix and consider 
existence of corresponding graphs. 
As we have seen, the model $\BA/\BB/\BC/\BD$ 
can relate to the $4$-cycle. For the model $\BA\BB/\BC/\BD$, 
we introduce imaginary factor $\BE = \BA\BB$ and consider the main effect model
for the $2^{5-2}$ design defined by
\[
\BA\BB\BC\BD = \BA\BB\BE\ (= \BC\BD\BE) = \BI
\]
as we have seen in the last of Section \ref{sec:mcmc}. This corresponds to
the graph with $5$ edges as follows.
\begin{center}
\begin{picture}(55,70)(0,0)
\put(0,15){\line(1,0){45}}
\put(0,15){\line(0,1){45}}
\put(45,15){\line(0,1){45}}
\put(0,60){\line(1,0){45}}
\put(0,60){\line(1,-1){45}}
\put(0,15){\circle*{5}}
\put(0,60){\circle*{5}}
\put(45,15){\circle*{5}}
\put(45,60){\circle*{5}}
\put(-15,35){$\BA$}
\put(18,0){$\BB$}
\put(15,25){$\BE$}
\put(18,65){$\BC$}
\put(53,35){$\BD$}
\end{picture}
\end{center}
Similarly, for the model 
$\BA\BB/\BA\BC/\BD$, introducing imaginary factors $\BE=\BA\BB$ and 
$\BF=\BA\BC$ and consider the main effect model for the $2^{6-3}$ design 
defined by
\[
\BA\BB\BC\BD = \BA\BB\BE = \BA\BC\BF = \BI,
\]
we see the corresponding graph is $K_4$.

For $2^{5-2}_{\rm III}$ design of $\BD = \BA\BB, \BE=\BA\BC$, 
we can consider only one additional interaction effect to the 
main effect model.
Therefore there are the following $2$ models to be considered.
\begin{itemize}
\item Main effect model:\\
We write it as $\BA/\BB/\BC/\BD/\BE$.
\item Models with all the main effects and $1$ two-factor interaction effect:\\
Note that we cannot consider models including, say, $\BA\times \BB$, because
it is confounded to the main effect of $\BD$. 
Without loss of generality, we consider $\BB\times \BC$ as the
interaction effect, with we write 
$\BA/\BB\BC/\BD/\BE$.
\end{itemize}
For these models, we can construct model matrix and consider existence of
corresponding graphs. The model $\BA/\BB/\BC/\BD/\BE$ can relate to the graph
as follows.
\begin{center}
\begin{picture}(55,70)(0,0)
\put(0,15){\line(1,0){45}}
\put(0,15){\line(0,1){45}}
\put(45,15){\line(0,1){45}}
\put(0,60){\line(1,0){45}}
\put(0,60){\line(1,-1){45}}
\put(0,15){\circle*{5}}
\put(0,60){\circle*{5}}
\put(45,15){\circle*{5}}
\put(45,60){\circle*{5}}
\put(-15,35){$\BD$}
\put(18,0){$\BB$}
\put(15,25){$\BA$}
\put(18,65){$\BC$}
\put(53,35){$\BE$}
\end{picture}
\end{center}
We see the model $\BA/\BB\BC/\BD/\BE$ can relate 
to $K_4$ by introducing imaginary
factor $\BF=\BB\BC$.  

Finally, for $2^{6-3}_{\rm III}$ design of $\BD=\BA\BB, \BE=\BA\BC, 
\BF=\BB\BC$, only the main effect model can be considered, which can be
relate to $K_4$.

From the above considerations, we have complete classification of all the
hierarchical models for the designs with $8$ runs in 
Table \ref{tbl:list-8-runs-designs}.
We summarize the results in Table \ref{tbl:models-for-8runs} and 
Table \ref{tbl:graphs-for-8runs}.

\begin{table*}
\begin{center}
\caption{Models for the designs in Table \ref{tbl:list-8-runs-designs}}
\label{tbl:models-for-8runs}
\begin{tabular}{lclr}\hline
Design & Num. of parameters & Model & Index\\ \hline
$2^{4-1}_{\rm IV}$ & $5$ & $\BA/\BB/\BC/\BD$ & [1]\\
                   & $6$ & $\BA\BB/\BC/\BD$ & [2]\\
                   & $7$ & $\BA\BB/\BA\BC/\BD$ & [3]\\ \hline
$2^{5-2}_{\rm III}$ & $6$ & $\BA/\BB/\BC/\BD/\BE$ & [4]\\
                   & $7$ & $\BA/\BB\BC/\BD/\BE$ & [5]\\ \hline
$2^{6-3}_{\rm III}$ & $7$ & $\BA/\BB/\BC/\BD/\BE/\BF$ & [6]\\ \hline
\end{tabular}
\end{center}
\end{table*}
\begin{table*}[h]
\begin{center}
\caption{Graphs corresponding to the models in Table \ref{tbl:models-for-8runs}}
\label{tbl:graphs-for-8runs}
\begin{tabular}{cl}\hline
Graph & Models\\ \hline
{\begin{picture}(35,35)(0,0)
\put(0,0){\line(1,0){30}}
\put(0,0){\line(0,1){30}}
\put(30,0){\line(0,1){30}}
\put(0,30){\line(1,0){30}}
\put(0,00){\circle*{5}}
\put(0,30){\circle*{5}}
\put(30,0){\circle*{5}}
\put(30,30){\circle*{5}}
\end{picture}} & [1]\\
{\begin{picture}(35,35)(0,0)
\put(0,0){\line(1,0){30}}
\put(0,0){\line(0,1){30}}
\put(30,0){\line(0,1){30}}
\put(0,30){\line(1,0){30}}
\put(0,30){\line(1,-1){30}}
\put(0,00){\circle*{5}}
\put(0,30){\circle*{5}}
\put(30,0){\circle*{5}}
\put(30,30){\circle*{5}}
\end{picture}} & [2][4]\\
{\begin{picture}(35,35)(0,0)
\put(0,0){\line(1,0){30}}
\put(0,0){\line(0,1){30}}
\put(30,0){\line(0,1){30}}
\put(0,30){\line(1,0){30}}
\put(0,30){\line(1,-1){30}}
\put(0,0){\line(1,1){30}}
\put(0,00){\circle*{5}}
\put(0,30){\circle*{5}}
\put(30,0){\circle*{5}}
\put(30,30){\circle*{5}}
\end{picture}} & [3][5][6]\\ \hline
\end{tabular}
\end{center}
\end{table*}
Next we consider models for the designs with $16$ runs in a similar way.
The most frequently used designs with $16$ runs are given 
Section 4 of \cite{wu-hamada-2000}, which we show in 
Table \ref{tbl:list-16-runs-designs}.
\begin{table*}
\begin{center}
\caption{$2^{p-q}$ fractional factorial designs with $16$ runs ($p - q = 4$)}
\label{tbl:list-16-runs-designs}
\begin{tabular}{ccl}\hline
Number of factors $p$ & Resolution & Design Generators\\ \hline
5 & V   & $\BE = \BA\BB\BC\BD$ \\
6 & IV  & $\BE = \BA\BB\BC, \BF = \BA\BB\BD$\\
7 & IV  & $\BE = \BA\BB\BC, \BF = \BA\BB\BD, \BG = \BA\BC\BD$\\
8 & IV  & $\BE = \BA\BB\BC, \BF = \BA\BB\BD, \BG = \BA\BC\BD$\\
  &     & $\BH = \BB\BC\BD$\\
9 & III & $\BE = \BA\BB\BC, \BF = \BA\BB\BD, \BG = \BA\BC\BD$\\
  &     & $\BH = \BB\BC\BD, \BJ = \BA\BB\BC\BD$\\
10 & III & $\BE = \BA\BB\BC, \BF = \BA\BB\BD, \BG = \BA\BC\BD$\\
   &     & $\BH = \BB\BC\BD, \BJ = \BA\BB\BC\BD, \BK = \BC\BD$\\ \hline
\end{tabular}
\end{center}
\end{table*}
For the designs in Table \ref{tbl:list-16-runs-designs}, we consider 
possible models. For the designs with $16$ runs, 
we can test models with less than $16$ parameters. However, 
the models with parameters more than $11\ (= 10 + 1)$ cannot relate
to graphs obviously because there are $10$ edges in $K_5$. 
Therefore we only consider models with at most $11$ parameters.

Moreover, we can use the fact that if some model has a corresponding 
graph, its each submodel also has a corresponding graph. 
We can confirm this fact by tracing the arguments of imaginary factors
reversely. 
Suppose some model ${\cal M}$ with the interaction factor 
$\BA_1\times\cdots\times \BA_s$ has a corresponding graph $G$
with $q$ dimensional cycle space.
The graph $G$ is constructed by introducing
imaginary factor $\BX$ to consider the cycle 
$\{\BX,\BA_1\ldots,\BA_s\}$. 
It then follows that there exists a spanning tree of $G$ such that
$\{\BX,\BA_1\ldots,\BA_s\}$ is a fundamental cycle.
Then,
deleting the edge $\BX$ yields a graph with $q-1$ dimensional cycle space, 
which is the corresponding graph of the 
model ${\cal M}\setminus \{\BA_1\times\cdots\times \BA_s\}$. 
In particular, if the main effect model does not have a corresponding graph,
all models with interaction effects also do not have a corresponding graph for
this design. 
For the designs of Table \ref{tbl:list-16-runs-designs}, we see that 
the main effect models for $2^{7-3}_{{\rm IV}}$, 
$2^{8-4}_{{\rm IV}}$, $2^{9-5}_{{\rm III}}$ and $2^{10-6}_{{\rm III}}$ do 
not have corresponding graphs. 
(For example, 
if a $2^{p-q} = 2^4$ design of resolution IV
corresponds to a finite graph by the relation (\ref{eqn:def-X_e}),
then we have $p \leq \lfloor (4+1)^2/4 \rfloor = 6$
by Proposition \ref{trifreegraph}.)
Therefore we consider the models 
for $2^{5-1}_{{\rm V}}$ and $2^{6-2}_{{\rm IV}}$ designs. 
The distinct models for these designs
are given in Table \ref{tbl:models-for-16runs-5} and
Table \ref{tbl:models-for-16runs-6}.
For the models with no corresponding graphs, only minimal models are included
in Table \ref{tbl:models-for-16runs-5} and
Table \ref{tbl:models-for-16runs-6}.
For example, the model $\BA\BB/\BA\BC/\BA\BD/\BE$ of 
the $2^{5-1}_{{\rm V}}$-design does not have a corresponding graph.
This is a minimal model in the sense that any submodel of it, i.e., 
$\BA\BB/\BA\BC/\BD/\BE$,  
$\BA\BB/\BC/\BD/\BE$ and $\BA/\BB/\BC/\BD/\BE$, has a corresponding graph. 
From the consideration above, we see all models including it, i.e., 
 $\BA\BB/\BA\BC/\BA\BD/\BA\BE$ or
 $\BA\BB/\BA\BC/\BB\BC/\BA\BD/\BA\BE$, for example, also do not have a
corresponding graph.

\begin{table*}
\begin{center}
\caption{Models for the $2^{5-1}_{{\rm V}}$-design of $\BE=\BA\BB\BC\BD$}
\label{tbl:models-for-16runs-5}
\begin{tabular}{clc}\hline
 Num. of parameters & Model & Index\\ \hline
 $6$ & $\BA/\BB/\BC/\BD/\BE$ & [5-1]\\
     $7$ & $\BA\BB/\BC/\BD/\BE$ & [5-2]\\
     $8$ & $\BA\BB/\BA\BC/\BD/\BE$ & [5-3]\\
         & $\BA\BB/\BC\BD/\BE$ & [5-4]\\
     $9$ & $\BA\BB/\BA\BC/\BB\BD/\BE$ & [5-5]\\
         & $\BA\BB/\BA\BC/\BD\BE$ & [5-6]\\
         & $\BA\BB/\BA\BC/\BA\BD/\BE$ & (no graph)\\ 
         & $\BA\BB/\BA\BC/\BB\BC/\BD/\BE$ & (no graph)\\
     $10$& $\BA\BB/\BA\BC/\BB\BD/\BC\BE$ & [5-7]\\
         & $\BA\BB/\BA\BC/\BB\BD/\BC\BD$ & (no graph)\\
     $11$& $\BA\BB/\BA\BC/\BB\BD/\BC\BE/\BD\BE$ & [5-8]\\ \hline
\end{tabular}
\end{center}
\end{table*}

\begin{table*}
\begin{center}
\caption{Models for the $2^{6-2}_{{\rm IV}}$-design of $\BE=\BA\BB\BC, 
\BF=\BA\BB\BD$}
\label{tbl:models-for-16runs-6}
\begin{tabular}{clr}\hline
Num. of parameters & Model & Index\\ \hline
 $7$ & $\BA/\BB/\BC/\BD/\BE/\BF$ & [6-1]\\
 $8$ & $\BA\BB/\BC/\BD/\BE/\BF$ & [6-2]\\
     & $\BA\BC/\BB/\BD/\BE/\BF$ & [6-3]\\
 $9$ & $\BA\BB/\BA\BC/\BD/\BE/\BF$ & [6-4]\\
     & $\BA\BB/\BC\BD/\BE/\BF$ & [6-5]\\
 $10$& $\BA\BB/\BA\BC/\BA\BD/\BE/\BF$ & [6-6]\\
     & $\BA\BB/\BA\BD/\BB\BC/\BE/\BF$ & [6-7]\\
     & $\BA\BB/\BA\BC/\BC\BD/\BE/\BF$ & [6-8]\\
     & $\BA\BB/\BA\BC/\BD\BE/\BF$ & [6-9]\\
     & $\BA\BC/\BB\BD/\BE\BF$ & [6-10]\\ 
     & $\BA\BB/\BA\BC/\BA\BE/\BD/\BF$ & (no graph)\\
     & $\BA\BB/\BA\BC/\BB\BC/\BD/\BE/\BF$ & (no graph)\\
$11$ & $\BA\BB/\BA\BC/\BA\BD/\BC\BD/\BE/\BF$ & [6-11]\\
     & $\BA\BD/\BD\BE/\BD\BF/\BB\BC$ & [6-12]\\
     & $\BA\BD/\BB\BC/\BC\BF/\BD\BF/\BE$ & [6-13]\\
     & $\BA\BC/\BA\BD/\BC\BD/\BC\BF/\BB/\BE$ & (no graph)\\
     & $\BA\BB/\BA\BC/\BC\BF/\BC\BD/\BE$ & (no graph)\\
     & $\BA\BC/\BB\BD/\BC\BD/\BC\BF/\BE$ & (no graph)\\
     & $\BA\BC/\BB\BC/\BA\BD/\BA\BF/\BE$ & (no graph)\\
     & $\BA\BB/\BA\BC/\BA\BD/\BC\BF/\BE$ & (no graph)\\
     & $\BA\BC/\BA\BD/\BB\BC/\BD\BF/\BE$ & (no graph)\\
     & $\BA\BB/\BB\BC/\BA\BD/\BE\BF$ & (no graph)\\
     & $\BA\BD/\BB\BC/\BB\BD/\BE\BF$ & (no graph)\\
     & $\BA\BD/\BB\BC/\BB\BE/\BD\BF$ & (no graph)\\
     & $\BA\BD/\BA\BF/\BB\BC/\BB\BE$ & (no graph)\\ \hline
\end{tabular}
\end{center}
\end{table*}
For these models, we can construct model matrix and consider existence of
corresponding graphs.
The results are shown in Table \ref{tbl:graphs-for-16runs}.

\begin{table*}
\begin{center}
\caption{Graphs corresponding to the models in Tables \ref{tbl:models-for-16runs-5}
and \ref{tbl:models-for-16runs-6}}
\label{tbl:graphs-for-16runs}
\begin{tabular}{cl}\hline
Graph & Models\\ \hline
{$\begin{array}{c}G_1=\\ \\ \end{array}$\begin{picture}(70,35)(0,0)
\put(0,0){\line(1,0){30}}
\put(0,0){\line(0,1){30}}
\put(30,0){\line(2,1){30}}
\put(30,30){\line(2,-1){30}}
\put(0,30){\line(1,0){30}}
\put(0,00){\circle*{5}}
\put(0,30){\circle*{5}}
\put(30,0){\circle*{5}}
\put(30,30){\circle*{5}}
\put(60,15){\circle*{5}}
\end{picture}} & [5-1]\\ \hline
{$\begin{array}{c}G_2=\\ \\ \end{array}$\begin{picture}(70,35)(0,0)
\put(0,0){\line(1,0){30}}
\put(0,0){\line(0,1){30}}
\put(30,0){\line(0,1){30}}
\put(30,0){\line(2,1){30}}
\put(30,30){\line(2,-1){30}}
\put(0,30){\line(1,0){30}}
\put(0,00){\circle*{5}}
\put(0,30){\circle*{5}}
\put(30,0){\circle*{5}}
\put(30,30){\circle*{5}}
\put(60,15){\circle*{5}}
\end{picture}} & [5-2]\\
{$\begin{array}{c}G_3=\\ \\ \end{array}$\begin{picture}(70,35)(0,0)
\put(0,0){\line(1,0){30}}
\put(0,0){\line(0,1){30}}
\put(30,0){\line(0,1){30}}
\put(30,0){\line(2,1){30}}
\put(0,30){\line(4,-1){60}}
\put(0,30){\line(1,0){30}}
\put(0,00){\circle*{5}}
\put(0,30){\circle*{5}}
\put(30,0){\circle*{5}}
\put(30,30){\circle*{5}}
\put(60,15){\circle*{5}}
\end{picture}} & [6-1]\\ \hline
{$\begin{array}{c}G_4=\\ \\ \end{array}$\begin{picture}(70,35)(0,0)
\put(0,0){\line(1,0){30}}
\put(0,0){\line(0,1){30}}
\put(0,30){\line(4,-1){60}}
\put(30,0){\line(0,1){30}}
\put(30,0){\line(2,1){30}}
\put(30,30){\line(2,-1){30}}
\put(0,30){\line(1,0){30}}
\put(0,00){\circle*{5}}
\put(0,30){\circle*{5}}
\put(30,0){\circle*{5}}
\put(30,30){\circle*{5}}
\put(60,15){\circle*{5}}
\end{picture}} & [5-3]\\
{$\begin{array}{c}G_5=\\ \\ \end{array}$\begin{picture}(70,35)(0,0)
\put(0,0){\line(1,0){30}}
\put(0,0){\line(0,1){30}}
\put(30,0){\line(0,1){30}}
\put(30,0){\line(2,1){30}}
\put(30,30){\line(2,-1){30}}
\put(0,30){\line(1,0){30}}
\put(0,30){\line(1,-1){30}}
\put(0,00){\circle*{5}}
\put(0,30){\circle*{5}}
\put(30,0){\circle*{5}}
\put(30,30){\circle*{5}}
\put(60,15){\circle*{5}}
\end{picture}} & [5-4]\\
{$\begin{array}{c}G_6=\\ \\ \end{array}$\begin{picture}(70,35)(0,0)
\put(0,0){\line(1,0){30}}
\put(0,0){\line(0,1){30}}
\put(30,0){\line(0,1){30}}
\put(30,0){\line(2,1){30}}
\put(0,30){\line(4,-1){60}}
\put(0,30){\line(1,0){30}}
\put(0,30){\line(1,-1){30}}
\put(0,00){\circle*{5}}
\put(0,30){\circle*{5}}
\put(30,0){\circle*{5}}
\put(30,30){\circle*{5}}
\put(60,15){\circle*{5}}
\end{picture}} & [6-2]\\
{$\begin{array}{c}G_7=\\ \\ \end{array}$\begin{picture}(70,35)(0,0)
\put(0,0){\line(1,0){30}}
\put(0,0){\line(0,1){30}}
\put(30,0){\line(0,1){30}}
\put(30,0){\line(2,1){30}}
\put(0,30){\line(4,-1){60}}
\put(0,30){\line(1,0){30}}
\put(0,0){\line(1,1){30}}
\put(0,00){\circle*{5}}
\put(0,30){\circle*{5}}
\put(30,0){\circle*{5}}
\put(30,30){\circle*{5}}
\put(60,15){\circle*{5}}
\end{picture}} & [6-3]\\ \hline
{$\begin{array}{c}G_8=\\ \\ \end{array}$\begin{picture}(70,35)(0,0)
\put(0,0){\line(1,0){30}}
\put(0,0){\line(0,1){30}}
\put(0,0){\line(1,1){30}}
\put(0,30){\line(4,-1){60}}
\put(30,0){\line(2,1){30}}
\put(30,30){\line(2,-1){30}}
\put(0,30){\line(1,0){30}}
\put(0,30){\line(1,-1){30}}
\put(0,00){\circle*{5}}
\put(0,30){\circle*{5}}
\put(30,0){\circle*{5}}
\put(30,30){\circle*{5}}
\put(60,15){\circle*{5}}
\end{picture}} & [5-5]\\
{$\begin{array}{c}G_9=\\ \\ \end{array}$\begin{picture}(70,35)(0,0)
\put(0,0){\line(1,0){30}}
\put(0,0){\line(0,1){30}}
\put(0,0){\line(1,1){30}}
\put(30,0){\line(0,1){30}}
\put(30,0){\line(2,1){30}}
\put(30,30){\line(2,-1){30}}
\put(0,30){\line(1,0){30}}
\put(0,30){\line(1,-1){30}}
\put(0,00){\circle*{5}}
\put(0,30){\circle*{5}}
\put(30,0){\circle*{5}}
\put(30,30){\circle*{5}}
\put(60,15){\circle*{5}}
\end{picture}} & [5-6][6-4][6-5]\\ \hline 
{$\begin{array}{c}G_{10}=\\ \\ \end{array}$\begin{picture}(70,35)(0,0)
\put(0,0){\line(1,0){30}}
\put(0,0){\line(0,1){30}}
\put(0,0){\line(1,1){30}}
\put(0,30){\line(4,-1){60}}
\put(0,0){\line(4,1){60}}
\put(30,0){\line(2,1){30}}
\put(30,30){\line(2,-1){30}}
\put(0,30){\line(1,0){30}}
\put(0,30){\line(1,-1){30}}
\put(0,00){\circle*{5}}
\put(0,30){\circle*{5}}
\put(30,0){\circle*{5}}
\put(30,30){\circle*{5}}
\put(60,15){\circle*{5}}
\end{picture}} & [5-7][6-6][6-7][6-8][6-9][6-19]\\ \hline
{$\begin{array}{c}G_{11}=\\ \\ \end{array}$\begin{picture}(70,35)(0,0)
\put(0,0){\line(1,0){30}}
\put(0,0){\line(0,1){30}}
\put(0,0){\line(1,1){30}}
\put(0,30){\line(4,-1){60}}
\put(0,0){\line(4,1){60}}
\put(30,0){\line(2,1){30}}
\put(30,0){\line(0,1){30}}
\put(30,30){\line(2,-1){30}}
\put(0,30){\line(1,0){30}}
\put(0,30){\line(1,-1){30}}
\put(0,00){\circle*{5}}
\put(0,30){\circle*{5}}
\put(30,0){\circle*{5}}
\put(30,30){\circle*{5}}
\put(60,15){\circle*{5}}
\end{picture}} & [5-8][6-11][6-12][6-13]\\ \hline
\end{tabular}
\end{center}
\end{table*}

%% file: section3.tex
\section{Application}
\label{sec:application}

In this section, we apply known results on cut ideals
to the regular two-level fractional factorial designs.
First we study fundamental facts on cut ideals appearing
in \cite{Sturmfels-Sullivant-2008}.
Let $G_1 = (V_1, E_1)$ and 
$G_2 = (V_2, E_2)$ be graphs such that $V_1 \cap V_2$ is a clique of both graphs.
The new graph $G = G_1 \sharp G_2$ with the vertex set $V = V_1 \cup V_2$
and edge set $E=E_1 \cup E_2$ is called {\em $k$-sum} of $G_1$ and $G_2$
along $V_1 \cap V_2$ if the cardinality of $V_1 \cap V_2$ is $k+1$.
In this section, we only consider 0, 1, 2-sums.
Let
$$
{\bf f} = 
\prod_{i=1}^d q_{A_i | B_i} -\prod_{i=1}^d q_{C_i | D_i}
$$
be a binomial in $I_{G_1}$ of degree $d$.
Since $V_1 \cap V_2$ is a clique of $G_1$ and $|V_1 \cap V_2| \leq 3$,
we may assume that 
$A_i \cap V_1 \cap V_2 = C_i \cap V_1 \cap V_2$ for all $i$.
For any ordered list $EF$ of $d$ partitions of $V_2 \setminus  V_1$,
$$
EF = (E_1|F_1, E_2 | F_2 , \ldots, E_d | F_d),
$$
we define the binomial in $I_G$ of degree $d$ by
$$
{\bf f}^{EF} =
\prod_{i=1}^d q_{A_i \cup E_i | B_i \cup F_i} -
\prod_{i=1}^d q_{C_i \cup E_i| D_i \cup F_i}
\in 
I_G.
$$
If ${\bf F}$ is a set of binomials in $I_{G_1}$, then
we define
$$
{\rm Lift}({\bf F}) = 
\left\{
\left.
{\bf f}^{EF} \ \right| \ 
{\bf f} \in {\bf F}, \ EF = \{ E_i | F_i\}_{i=1}^{\deg {\bf f}}
\right\}.
$$
On the other hand,
let ${\rm Quad} (G_1, G_2)$ be the set of all quadratic binomials
$$
q_{ A \cup C_1 \cup C_2 | B \cup D_1 \cup D_2 }
q_{ A \cup E_1 \cup E_2 | B \cup F_1 \cup F_2 }
-
q_{ A \cup E_1 \cup C_2 | B \cup F_1 \cup D_2 }
q_{ A \cup C_1 \cup E_2 | B \cup D_1 \cup F_2 }
$$
where
\begin{itemize}
\item
$A|B$ is an unordered partition of $V_1 \cap V_2$;
\item
$C_1|D_1$ and $E_1|F_1$ are ordered partitions of $V_1 \setminus V_2$;
\item
$C_2|D_2$ and $E_2|F_2$ are ordered 
partitions of $V_2 \setminus V_1$.
\end{itemize}
Then, the following is known:

\begin{proposition}[\cite{Sturmfels-Sullivant-2008}]
\label{toricfiber}
Let $G = G_1 \sharp G_2$ be a $0$, $1$ or $2$-sum of $G_1$ and $G_2$
and
let ${\bf F}_i$ be a set of binomial generators of $I_{G_i}$ for $i = 1,2$.
Then $I_G$ is generated by
$$
{\bf M}=
{\rm Lift}({\bf F}_1) \cup {\rm Lift}({\bf F}_2) \cup {\rm Quad} (G_1, G_2).
$$
Moreover, if ${\bf F}_i$ is a Gr\"obner basis of $I_{G_i}$ for $i = 1,2$,
then there exists a monomial order such that ${\bf M}$ is  a Gr\"obner basis 
of $I_G$.
\end{proposition}

\begin{example}
The graph $G_9$ in Table \ref{tbl:graphs-for-16runs} is a
1-sum of the complete graph $K_4$ and a cycle $C_3$ of length 3.
Since $I_{K_4}$ is a principal ideal in Example \ref{example:K4} and
$I_{C_3}$ is a zero ideal, $I_{G_9}$ is generated by the binomials
of degree 2 and 4.
On the other hand,
the graph $G_{10}$ in Table \ref{tbl:graphs-for-16runs} is a
2-sum of the complete graphs $K_4$ and $K_4$.
Thus, $I_{G_{10}}$ is generated by the binomials
of degree 2 and 4, too.
\end{example}

A graph $G$ is called a {\em ring graph} if 
$G$ is obtained by 0/1-sums of cycles and edges.
It is known that ring graphs have no $K_4$ minor.

\begin{proposition}[\cite{Nagel-Petrovic-2009}]
\label{ringgraph}
If $G$ is a ring graph, then
$I_G$ has a quadratic Gr\"obner basis.
\end{proposition}

\begin{example}
It is easy to see that
the graph $G_i$ in Table \ref{tbl:graphs-for-16runs} is 
a ring graph if and only if $i \in \{
1,2,5,6
\}$.
Thus the cut ideal $I_{G_i}$ has a quadratic Gr\"obner basis
for $i \in \{
1,2,5,6
\}$.
\end{example}

Let  $e =\{i,j\} \in E$ be an edge of a graph $G=(V,E)$.
Then, the new graph $G \setminus e := (V,E\setminus\{e\})$ is called
the graph obtained from $G$ by {\em deleting} $e$.
On the other hand, the new graph $G / e$ obtained by 
the procedure
\begin{itemize}
\item[(i)]
Identify the vertices $i$ and $j$;
\item[(ii)]
Delete the multiple edges that may be created while (i);
\end{itemize}
is called the graph obtained from $G$ by {\em contracting} $e$.
A graph $H$ is said to be a {\em minor} of $G$
if it can be obtained from $G$ by a sequence of deletions and/or contractions
of edges (and deletions of vertices).
The following theorem is conjectured by 
Sturmfels--Sullivant
\cite{Sturmfels-Sullivant-2008}
and proved by Engstr\"om:

\begin{proposition}[\cite{Engstrom-2011}]
\label{Alex}
The toric ideal $I_G$ is generated by quadratic binomials if and only if 
$G$ has no $K_4$ minor.
\end{proposition}

\begin{example}
It is easy to see that
the graph $G_i$ in Table \ref{tbl:graphs-for-16runs}
does not have $K_4$ minor if and only if $i \in \{
1,2,3,5,6
\}$.
Thus the cut ideal $I_{G_i}$ is generated by quadratic binomials
if and only if $i \in \{
1,2,3,5,6
\}$.
\end{example}

Let $G$ be a graph with vertex set $V  = \{1,\ldots,n\}$ and edge set $E$.
The {\em suspension} of the graph $G$ is the new graph $\widehat{G}$
whose vertex set equals $ V \cup \{n+1\}$ and whose edge set
equals $E \cup \{\{ i,n+1\} \ | \ i \in V\} $.
It is known \cite{Sturmfels-Sullivant-2008} that
the toric ideal of the binary graph model of $G$ equals to the cut ideal 
$I_{\widehat{G}}$
of $\widehat{G}$.

\begin{proposition}[\cite{Kral-Norine-Pangrac-2010}]
Let $\widehat{G}$ be the suspension of $G$.
Then
$I_{\widehat{G}}$
is generated by binomials of degree $\leq 4$ if and only if 
$G$ has no $K_4$ minor.
\end{proposition}

\begin{example}
The graph $G_8$ in Table \ref{tbl:graphs-for-16runs} is
the suspension of a cycle $C_4$ of length 4.
Since $C_4$ has no $K_4$ minor, $I_{G_8}$ is generated by the binomials
of degree $\leq 4$.
On the other hand,
the graph $G_{11} =K_5$ in Table \ref{tbl:graphs-for-16runs} is
the suspension of $K_4$.
Thus, $I_{G_{11}}$ is not generated by the binomials
of degree $\leq 4$.
\end{example}

We now apply these known results to our problem.

\begin{theorem}
\label{application}
Let $D$ be a regular fractional factorial design with 
at most two defining relations.
Then there exists a connected graph $G = (V, E)$
such that
\begin{itemize}
\item
$D$ is the design matrix
of $|E|$ factors with $2^{|V|-1}$ runs defined by 
(\ref{eqn:def-X_e}).
\item
$I_G$ is generated by quadratic binomials.
\end{itemize}
Moreover, if  $D$ has exactly one defining relations,
then
$I_G$ has a quadratic Gr\"obner basis.
\end{theorem}

\begin{proof}
Let $D$ be a regular $2^{p-1}$ fractional factorial design with 
the defining relation
$$
{\bf A}_1 {\bf A}_2 \cdots {\bf A}_r  ={\bf I}.
$$
Let $G=(V,E)$ be a graph on the vertex set 
$V=\{1,2,\ldots,p\}$ with the edge set
$$
E=\{
12, 23, \ldots, (r-1) r, 1r,
r (r+1) , (r+1) (r+2) ,\ldots, (p-1)p
\}.
$$
Then, it is easy to see that
$D$ is the design matrix
of $|E|$ factors with $2^{|V|-1}$ runs defined by 
(\ref{eqn:def-X_e}).
Since $G$ is a
ring graph, 
$I_G$ has a quadratic Gr\"obner basis
by Proposition \ref{ringgraph}.

Let $D$ be a regular $2^{p-2}$ fractional factorial design with 
the defining relation
$$
{\bf A}_1 {\bf A}_2 \cdots {\bf A}_r  
{\bf B}_1 {\bf B}_2 \cdots {\bf B}_s  
=
{\bf B}_1 {\bf B}_2 \cdots {\bf B}_s
{\bf C}_1 {\bf C}_2 \cdots {\bf C}_t
  =
{\bf I},
$$
where $0 \leq s \in {\mathbb Z}$ and $1 \leq r,t \in {\mathbb Z}$.
Let $G=(V,E)$ be 
the graph in Figure \ref{fig:deftwo}.
\begin{figure*}[h]
\begin{center}
\includegraphics[scale=.7]{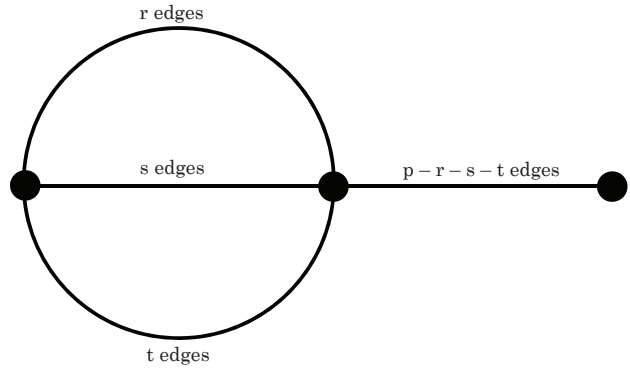}
\caption{A graph for a regular $2^{p-2}$ fractional factorial design}
\label{fig:deftwo}
\end{center}
\end{figure*}
It then follows that 
$D$ is the design matrix
of $|E|$ factors with $2^{|V|-1}$ runs defined by 
(\ref{eqn:def-X_e}).
It is easy to see that 
$G$ has no $K_4$ minor.
Hence, by virtue of Proposition \ref{Alex}, $I_G$ is generated by quadratic binomials.
\end{proof}

\begin{remark}
{\em
Explicit description of binomials appearing in Theorem \ref{application}
is given in
\cite{Nagel-Petrovic-2009} and \cite{Engstrom-2011}.
The set of generators of $I_G$ consisting of quadratic binomials
for a regular $2^{p-1}$ fractional factorial design is studied in \cite{Aokinew}.
}
\end{remark}

%% file: bibliography.tex
\bibliographystyle{plain}